\documentclass[letterpaper, 10 pt, conference]{ieeeconf}  

\IEEEoverridecommandlockouts                              
\newlength{\bibitemsep}
\newlength{\bibparskip}
\let\oldthebibliography\thebibliography
\renewcommand\thebibliography[1]{%
  \oldthebibliography{#1}%
  \setlength{\parskip}{\bibitemsep}%
  \setlength{\itemsep}{\bibparskip}%
}

\usepackage{amsmath,amssymb,euscript,yfonts,psfrag,latexsym,dsfont,graphicx}
\usepackage{bbm,color,amstext,wasysym,subfig,parskip,balance}

\graphicspath{{./},{./figures/}}

\newtheorem{thm}{Theorem}

\newtheorem{lemma}[thm]{Lemma}

\usepackage{algorithmic}
\usepackage{caption}
\usepackage{algorithm}
\usepackage{subfig}
\usepackage[dvipsnames]{xcolor}
\usepackage{ulem}

\newcommand{\mE}{{\mathbb E}}

\newcommand{\mR}{{\mathbb R}}

\newcommand{\cP}{{\mathcal P}}
\newcommand{\cQ}{{\mathcal Q}}

\usepackage{ulem}

\definecolor{grey}{rgb}{0.6,0.6,0.6}
\definecolor{lightgray}{rgb}{0.97,.99,0.99}

\title{\LARGE \bf An optimal control approach to particle filtering on Lie groups}

\author{Bo Yuan, Qinsheng Zhang, and Yongxin Chen
\thanks{This work was supported by NSF under grant 1942523 and 2008513.}
\thanks{Bo Yuan, Q. Zhang and Y.\ Chen are with the School of Aerospace Engineering,
Georgia Institute of Technology, Atlanta, GA; {\tt\small \{byuan48,qzhang419,yongchen\}@gatech.edu}}}

\begin{document}

\maketitle
\thispagestyle{empty}
\pagestyle{empty}

\begin{abstract}

We study the filtering problem over a Lie group that plays an important role in robotics and aerospace applications. We present a new particle filtering algorithm based on stochastic control. In particular, our algorithm is based on a duality between smoothing and optimal control. Leveraging this duality, we reformulate the smoothing problem into an optimal control problem, and by approximately solving it (using, e.g., iLQR) we establish a superior proposal for particle smoothing. Combining it with a suitable designed sliding window mechanism, we obtain a particle filtering algorithm that suffers less from sample degeneracy compared with existing methods. The efficacy of our algorithm is illustrated by a filtering problem over SO(3) for satellite attitude estimation.  
\end{abstract}

\section{INTRODUCTION}


The filtering problem to estimate the posterior distribution of the state of a dynamic system from noisy observations plays an essential role in applications in control engineering~\cite{AndMoo12}, robotics~\cite{DurBai06}, autonomy~\cite{BusVarLoh20} etc. For dynamics that has a Euclidean state space, many algorithms such as the extended Kalman filter (EKF), the unscented Kalman filter (UKF), particle filtering, have been proposed for the filtering task. In many applications in robotics and aerospace, the state space of the dynamical system is a manifold instead. The above filtering algorithms cannot directly be extended to the manifold setting due to the differences between the Euclidean spaces and general manifolds.

Designing filters directly evolving on manifolds gives better estimation of posteriors in most cases. For instance, for the agent's pose evolving in the three-dimensional space, the state variable belongs to the special orthogonal group, a Lie group SO(3). On this manifold, better convergence results from the Invariant Extended Kalman Filters have been shown in ~\cite{ZhaWuSon17,KoSonkim18,BarBon16,HasAboAbi21}.
However, these filters are sensitive to the initial state estimation and only locally stable. Alternatively, Lie Group Variational Filters~\cite{IzaSan16,SanIzaSam14}
derived from the d’Alembert’s Principle are shown to be almost-globally asymptotically stable. 


We present a particle filtering for dynamical systems on Lie groups \cite{ZhaTagMeh17}.
Our algorithm relies on the duality of filtering and stochastic optimal control,
which reformulates the smoothing problem into an optimal control problem~\cite{MitNew03,kimMeh20}.
If one can find the exact optimal solution for it, then the resulting filtering algorithm can achieve its theoretical maximum effectiveness. In practice, the computation of optimal control problems is demanding, thus we utilize a suboptimal solution instead
\cite{BouThe18a}. The particles in the resulting filtering/smoothing algorithm is generated by simulating the dynamics under suboptimal control. Path integral based on the Girsanov theorem~\cite{Oks03} and the Kallianpur-Striebel formula~\cite{Van07} is used to compute the weights of these particles. This work extends the Path Integral Particle Filter~\cite{ZhaTagChe21} to the Lie group setting.
We show with numerical experiments that our control-based filtering algorithm with the suboptimal control signal is more effective than that with a zero control signal. The latter is effectively the standard Sequential Importance Resampling.  

\section{notation and preliminaries}\label{sec:background}
In this section, we present the notation and background knowledge on Lie groups~\cite{SolDerAtc18} as well as the duality between smoothing and optimal control.

\subsection{Lie Groups and Lie Algebras}

We denote a Lie group by $G$, a set that is both a group and a $d$-dimensional differentiable manifold. Let $T_{g}G$ be the tangent space of $G$ at the element $g$. The associated Lie algebra $\mathfrak{g}$ is defined as $T_{e}G$ with $e$ being the identity element.

A vector field $V$ is a smooth mapping by which each element $g \in G$ corresponds to a vector $v \in T_{g}G$. Let $L_{g}$ denote the left multiplication by $g$, i.e., for all $h\in G$, $L_g(h) = gh$. It can be shown that $L_{g}$ is a diffeomorphism between $G$ and itself. Due to this diffeomorphism, for each vector $a \in \mathfrak{g}$, we can define a vector field $V_{a}$ that assigns each $g \in G$ to $dL_{g}(a)$, where $dL_{g}(a)$ is the shorthand notation of $(g\rho)^{\prime}(0)$ with $\rho(t)$, a curve on $G$, satisfying $\rho(0) = e$ and $\rho^{\prime}(0) = a$. To proceed, we denote the exponential mapping as $\exp$: $\mathfrak{g} \rightarrow G$. For every $a \in \mathfrak{g}$, $\exp(a) = \gamma(1)$, where $\gamma(t) \in G$ satisfies $\gamma^{\prime}(t) = dL_{\gamma(t)}(a)$ and $\gamma(0)=e$. As a local diffeomorphism between $G$ and $\mathfrak{g}$, the exponential mapping plays a crucial role in connecting them. In addition, we point out a Lie algebra $\mathfrak{g}$, as a vector space, is homeomorphic to $\mathbb{R}^{d}$. We denote the corresponding coordinate of $a \in \mathfrak{g}$ by $a^{\vee}$.


In many practical applications, Lie groups are assumed to be matrix Lie groups. A matrix Lie group is a closed subgroup of $\mathbb{GL}(n)$, the group consists of all $n \times n$ invertible matrices. In the case of matrix Lie groups, group multiplication corresponds to matrix multiplication, $dL_{g}(a)$ is the standard matrix multiplication for $g$ and $a$, i.e., $dL_{g}(a) = ga$, and the exponential mapping is matrix exponential. In what follows, we adopt $\exp$ to represent both the exponential mapping on  Lie groups and matrix exponential, as the exact meaning will be clear from context. Among these matrix Lie groups, the special orthogonal group SO(3), also named as 3D rotation group, is what we particularly care about. Formally, SO(3) is the group of all $ 3\times 3$ real orthogonal matrices with determinant 1. The Lie algebra of SO(3), denoted by $\mathfrak{so}(3)$, is all $3 \times 3$ skew-symmetric matrices. 

\subsection{Smoothing and Stochastic Optimal Control on $\mathbb{R}^d$}\label{section:smooth_control_R}
Recently, the duality between smoothing and optimal control inspires a class of filtering algorithms on $\mR^{d}$~\cite{ZhaTagChe21},~\cite{MitNew03},~\cite{kimMeh20},~\cite{Tod08}. The connection between the two areas is established as follows.
Consider a continuous dynamic system evolving on $\mathbb{R}^d$ in which the dynamics are represented by stochastic differential equations
	\begin{subequations}\label{eq:filteringdyn1}
	\begin{eqnarray}
	\label{eq:filteringdynx}	
	    dX_t &=& b(t,X_t) dt + \sigma(t,X_t) dW_t,\quad X_0\sim \nu_0, \\
	\label{eq:filteringdyny}	
	    dY_t &=& h(t,X_t)dt + \sigma_B dB_t,\quad Y_0 = 0.
	\end{eqnarray}
	\end{subequations}
Here, we denote the state variables by $X_t \in \mathbb{R}^d$ and the raw measurements by $Y_t \in \mathbb{R}^m$. The initial distribution of  $X_t$ is $\nu_0$. Moreover, $W_t$ and $B_t$ are two independent multidimensional Wiener processes. The task of the smoothing problem over $[0,T]$ is to estimate the posterior distribution of every $X_\tau$ given all the observations $\{Y_t, 0 \leq t\leq T\}$, i.e., $P(X_{\tau}|Y_t, 0 \leq t\leq T)$ for every $0 \leq \tau \leq T$. Consider another dynamics system with control signal $u_t$,
    \begin{equation}\label{eq:filteringdyn2}
        d\tilde X_t = b(t,\tilde X_t) dt + \sigma(t,\tilde X_t) (u_tdt+dW_t),\quad \tilde X_0\sim \pi_0,
    \end{equation}
where $\tilde X_t$ stands for the state variable under control signal $u_t$, and $\tilde X_0$ follows $\pi_0$. Note that the control signal  relies on both $t$ and $\tilde X_t$. We denote it by $u_t$ for brevity.

Denote the probability measure induced by dynamic \eqref{eq:filteringdynx} and dynamic \eqref{eq:filteringdyn2} by $\cP$ and $\tilde\cP$, respectively. We also let $\mathcal{Q}^Y$ be the posterior measure induced by \eqref{eq:filteringdyn1}. It has been shown that this smoothing problem can be reformulated to a stochastic optimal control problem by minimizing the KL divergence between $\tilde\cP$ and $\mathcal{Q}^Y$~\cite{ZhaTagChe21},~\cite{MitNew03},~\cite{kimMeh20},~\cite{Tod08}, i.e., 
    \begin{equation}\label{eq:KLdiv}
        \min_{\tilde\cP}~ {\rm KL} (\tilde\cP \| \mathcal{Q}^Y): = \mE_{\tilde\cP}(\log\frac{d\tilde\cP}{d\cQ^Y}).
    \end{equation}
It is worthwhile to summarize the derivation here, as our central proof is a generalization of it. The prementioned objective functions could be decomposed into two terms,
	\begin{equation}\label{eq:derive_R}
    \mE_{\tilde\cP}(\log\frac{d\tilde\cP}{d\cQ^Y}) =
    \mE_{\tilde\cP}(\log\frac{d\tilde\cP}{d\cP})
    -\mE_{\tilde\cP}(\log\frac{d\cQ^Y}{d\cP}) .
	\end{equation}
Under certain regularity conditions, one can adopt Pathwise Kallianpur-Striebel formula (See Proposition 1.4.2. in~\cite{Van07}) which gives
\begin{equation}
    \begin{aligned}\label{eq:KS_R}
    \frac{d\cQ^Y}{d\cP} \propto 
    ~&\exp\left\{-\int_0^T \frac{1}{2\sigma_B^2}\|h(t, \tilde X_t)\|^2 dt \right. \\
    &\hspace{-0.4cm}\left. -\int_0^T\frac{1}{\sigma_B^2}Y_t^{\prime} dh(t,\tilde X_t)+\frac{1}{\sigma_B^2}Y_T^{\prime}h(T,\tilde X_T)\right\}.
    \end{aligned}
\end{equation}

The other term can be dealt with by Girsanov theorem (See Theorem 8.6.5. in~\cite{Oks03}), and one has
\begin{align}
   \log\frac{d\tilde\cP}{d\cP} = \int_0^T \frac{1}{2}\|u_t\|^2 dt + u_t^{\prime}dW_t+ \log\frac{d\pi_0}{d\nu_0},
\end{align}
which infers
\begin{align}
		{\rm KL} (\tilde\cP\,\|\,\cP) = \mE \left\{\int_0^T \frac{1}{2}\|u_t\|^2 dt\right\}+ {\rm KL} (\pi_0\,\|\,\nu_0).
\end{align}
It follows the objective function of the derived optimal control problem is as follows
\begin{eqnarray}\nonumber
     \min_{u,\pi_0} &&\hspace{-0.65cm}{\rm KL} (\pi_0\,\|\,\nu_0)\!+\!\mE_{\tilde\cP}\left\{\int_0^T [\frac{1}{2}\|u_t\|^2+\frac{1}{2\sigma_B^2}\|h(t, \tilde X_t)\|^2] dt \right. \\ 
     &&\hspace{0.2cm}\left. +\int_0^T\frac{1}{\sigma_B^2}Y_t^{\prime} dh(t,\tilde X_t)-\frac{1}{\sigma_B^2}Y_T^{\prime}h(T,\tilde X_T)\right\}\label{eq:controlsmoothing}
\end{eqnarray}

In principle, once the optimal control strategy $u^{\star}_t$ and the optimal initial distribution $\pi_0^{\star}$ of \eqref{eq:controlsmoothing} are found, one can sample trajectories from dynamics \eqref{eq:filteringdyn2}, and the empirical distribution of these particles forms an approximation of the solution to smoothing problems.

\section{Path Integral Particle Filtering on Lie groups}
In this section, we present our main results on an optimal control inspired approach to particle filtering/smoothing over Lie groups.
\subsection{Smoothing as Control Problems for Lie Groups}
Consider the following dynamics for filtering on Lie groups
\begin{subequations}\label{eq:Lie_dynamics}
    \begin{align}
    		dg_t &= dL_{g_t}(\xi_{t})dt,\label{eq:LieDynamics1}\\
            d\xi_t^{\vee} &= b(t,g_t,\xi_t) dt + \sigma_W dW_t, \label{eq:LieDynamics2}\\
            dY_t &= h(t,g_t,\xi_t)dt + \sigma_B dB_t, \label{eq:LieDynamics3}\\
            (g_0,\xi_0^{\vee}) &\sim \nu_0, \quad Y_0 = 0.
    \end{align}
\end{subequations}
Here the velocity variable $\xi_t \in \mathfrak{g}$ stands for elements in Lie algebras, and the rotation variable $g_t \in G$ means elements in Lie groups. The output variable $Y_t \in \mathbb{R}^m$ represents raw observations, and the joint variable $(g_t,\xi_t^{\vee})$ stands for the state of the dynamic system. In addition, $W_t$ and $B_t$ represent two multidimensional independent Wiener processes as Section \ref{section:smooth_control_R}. We also denote the initial distribution of $(g_t,\xi_t^{\vee})$ by $\nu_0$, where $\nu_0$ is a well-defined distribution defined $G \times \mathfrak{g}$. 
The aim of filtering problem is to approximate the posterior distribution of $(g_{\tau},\xi_{\tau}^{\vee})$ given the observation up to $\tau$, i.e., $P\left((g_{\tau},\xi_{\tau}^{\vee}) | Y_t, 0\leq t \leq \tau \right)$. Note that the dynamic \eqref{eq:LieDynamics1} implies that, assuming $(g_0,\xi_0)$ is fixed, the state variable $\{(g_{\tau}, \xi_{\tau}^{\vee}),\,  0 \leq \tau \leq t\}$ is fully determined by  $\{\xi_{\tau}^{\vee},\,  0 \leq \tau \leq t\}$. Hence, loosely speaking, the measure of $(g_t,\xi_t^{\vee})$ is the marginal measure of $\xi_t^{\vee}$. For simplicity, we assume the parameters of \eqref{eq:Lie_dynamics} are sufficiently regular so that the solution to this system exists. Actually, the dynamics driven by \eqref{eq:LieDynamics1} and \eqref{eq:LieDynamics2} satisfy the Dol\'eans-Dade
and Protter’s equations, and the existence conditions of solutions for this type of equations have been studied in~\cite{MetPel80}. For \eqref{eq:LieDynamics3}, we assume $h_t$ is integrable almost surely.

In section \ref{section:smooth_control_R}, we have shown the smoothing problem on $\mathbb{R}^d$ can be reformulated as a stochastic control problem. In fact, the same framework can be applied on Lie groups. Let $\cP$  be the measure  of $(g_t,\xi_t^{\vee})$ given by \eqref{eq:LieDynamics1} and \eqref{eq:LieDynamics2}. Denote $\cQ^{Y}$ to be its posterior distribution. Moreover, we denote the state of the following It\^{o} process with control signal $u_t$ by $(\tilde g_t,\tilde\xi_t)$ and the associated measure by $\tilde\cP$.
\begin{subequations} \label{eq:Lie_controlled}
    \begin{align}
    		d\tilde g_t &= dL_{\tilde g_t}(\tilde \xi_t)dt,\\
            d\tilde\xi_t^{\vee} &= b(t,g_t,\tilde \xi_t) dt + \sigma_W (u_tdt+dW_t),\\
            (\tilde g_0,\tilde \xi_0) &\sim \pi_0.
    \end{align}
\end{subequations}

The next lemma reveals that the Kullback–Leibler divergence of $\tilde P $ and $\cP$ of \eqref{eq:Lie_dynamics} is in the same form as the $\mathbb{R}^d$ case. 
\begin{lemma}\label{lemma:Lie1}
Assume the Novikov's condition holds, i.e., ${\mE}[\exp(\frac{1}{2}\int_0^T \|u_t\|^2dt)] < \infty$, then
\begin{equation}\label{eq:Lie1}
		{\rm KL} (\tilde\cP\,\|\,\cP) = \mE \left\{\int_0^T \frac{1}{2}\|u_t\|^2 dt\right\}+ {\rm KL} (\pi_0\,\|\,\nu_0).
\end{equation}
\end{lemma}
\begin{proof}
In the standard Gisanov theorem, an elementary assumption is the dynamic is the state variable is in $\mathbb{R}^d$, which seems to be violated in our problem. However, the $\xi_t^{\vee}$-marginal measure satisfies this assumption and the Novikov's condition holds, so one can compute the Kullback–Leibler divergence of marginals as follows. Let $\cP_{\xi}$ and $ \tilde P_{\xi}$ be the $\xi_t^{\vee}$-marginals for $\cP$ and and $ \tilde P$, respectively. Then, assuming the initial values are given we have
\begin{align}
   \log\frac{d\tilde\cP_{\xi}}{d\cP_{\xi}} = \int_0^T \frac{1}{2}\|u_t\|^2 dt + u_t^{\prime}dW_t. \nonumber
\end{align}
As the sample path of $g_{[0,T]}$ is exactly a function of $\xi_{[0,T]}$, it follows
\begin{align}
   \log\frac{d\tilde P}{d\cP} &= \log\frac{d\tilde\cP_{\xi}}{d\cP_{\xi}} + \log\frac{d\pi_0}{d\nu_0}\nonumber \\
   &= \int_0^T \frac{1}{2}\|u_t\|^2 dt + u_t^{\prime}dW_t + \log\frac{d\pi_0}{d\nu_0},\nonumber
\end{align}
which yields the Kullback–Leibler divergence in Lemma \ref{lemma:Lie1}. Note that the Novikov's condition is already satisfied in most practical applications, so we can assume it is true as most previous literatures do.
\end{proof}
\begin{lemma}\label{lemma:Lie2}
Assume $h_t$ is an adapted and continuously differentiable process and the condition, $\int_0^T \|h_s\|^2ds < \infty$ for any $T$, holds almost surely. Then we have the Radon–Nikodym derivative, 
\begin{equation}
    \begin{aligned}\label{eq:Lie2}
    \frac{d\cQ^Y}{d\cP} \propto 
    ~&\exp\left\{-\int_0^T \frac{1}{2\sigma_B^2}\|h(t, \tilde g_t, \tilde \xi_t)\|^2 dt \right. \\
 &\hspace{-1cm}\left. -\int_0^T\frac{1}{\sigma_B^2}Y_t^{\prime} dh(t,\tilde g_t, \tilde \xi_t)+\frac{1}{\sigma_B^2}Y_T^{\prime}h(T,\tilde g_T, \tilde \xi_T)\right\}.
\end{aligned}
\end{equation}

\end{lemma}
\begin{proof}
In the original derivation of the Pathwise Kallianpur-Striebel formula (Proposition 1.4.2. in~\cite{Van07}), the dynamics of state variables are assumed to be an adapted c$\grave{a}$dl$\grave{a}$g process in a Polish space. In our setting, the state variable $(g_t,\xi_t)$ follows an continuous diffusion process in the Cartesian product space of $g_t$ and $\xi_t$, which does satisfy the prementioned assumption. One can check the assumption that $\int_0^T \|h_s\|^2ds < \infty$ for any $T$, holds almost surely, is the main assumption in~\cite{Van07}, which yields
\begin{equation}
    \begin{aligned}
    \frac{d\cQ^Y}{d\cP} \propto 
    &\exp\left\{-\int_0^T \frac{1}{2\sigma_B^2}\|h(t, \tilde g_t, \tilde \xi_t)\|^2 dt \right. \\
 &\hspace{-1cm}\left. +\int_0^T\frac{1}{\sigma_B^2}h^{\prime}(t,\tilde g_t, \tilde \xi_t) dY_t \right\}. \nonumber
\end{aligned}
\end{equation}
It is easy to verify that with the assumption that $h_t$ is adapted and continuously differentiable, the integral $\int h_t^{\prime}dY_t$ satisfies the stochastic integration by parts formula. Therefore, the Radon–Nikodym derivative in our case is in the same form as the one in~\cite{Van07}. 
\end{proof}

By plugging \eqref{eq:Lie1} and \eqref{eq:Lie2} into ${\rm KL} (\tilde\cP \| \mathcal{Q}^Y)$ in \eqref{eq:KLdiv}, the corresponding stochastic optimal control problem becomes
\begin{equation}
    \begin{aligned}\label{eq:costLie}
	\min_{u,\pi_0} &~ {\mE}\!\left\{\!\int_{0}^T\!\!\left [L(t,\tilde g_t,\tilde \xi_t)\!+\!\frac{1}{2}\|u_t\|^2\right] dt \!+\! \Phi(\tilde g_T,\tilde \xi_T)\right\} \\
	&+ {\rm KL} (\pi_0\,\|\,\nu_0),
\end{aligned}
\end{equation}
where
	\begin{subequations}\label{eq:ginference}
	\begin{align}\label{eq:ginference1}
		L(t,x,y)dt \!&=\! \frac{1}{2\sigma_B^2}\|h(t,x,y)\|^2 dt\!+\!\frac{1}{\sigma_B^2}Y_t^{\prime}dh(t,x,y) \\
        \Phi(x,y) &= -\frac{1}{\sigma_B^2}Y_T^{\prime}h(T,x,y). \label{eq:ginference2}
	\end{align}
	\end{subequations}
In~\cite{ZhaTagChe21}, the authors point out the optimization problem of this form can be solved separately. The optimal signal $u_t^{\star}$ depends only on $t$ and $X_t$, thus one can calculate $u_t^{\star}$ initially. It has been shown that $u_t^{\star}$ can be solved in the similar manner for $\mathbb{R}^d$~\cite{BouThe18b}.
After the computation of  $u_t^{\star}$, 
one can find that the optimal initial distribution $\pi_0^\star$ is the posterior distribution of $(g_0,\xi_0^{\vee})$ given all the observations $\{Y_t, 0\leq t \leq T\}$ which resembles the $\mR^{d}$ setting~\cite{ZhaTagChe21}.Thus, if the solution of \eqref{eq:costLie} exists, the minimum value is 0.
For the sake of completeness, we summarize our previous result in the following theorem.
\begin{thm}
Assume a solution to \eqref{eq:costLie} exists. 
and the conditions of Lemma \ref{lemma:Lie1} and Lemma \ref{lemma:Lie2} hold. For the dynamic system modeled by \eqref{eq:Lie_dynamics} and \eqref{eq:Lie_controlled}, denote the related optimal solution to \eqref{eq:costLie} by $u^{\star}_t$ and $\pi_0^{\star}$. Then the posterior measure ${\cQ^Y}$ is equal to the measure of dynamics $\tilde\cP$ driven by $u^{\star}_t$ and $\pi_0^{\star}$.
\end{thm}
\begin{proof}
The form of the corresponding stochastic optimal control problem follows from Lemma \ref{lemma:Lie1} and Lemma \ref{lemma:Lie2}. Since the dynamics $\tilde\cP$ driven by $u^{\star}_t$ and $\pi_0^{\star}$ satisfy ${\rm KL} (\tilde\cP\,\|\,{\cQ^Y}) = 0$, one can recover the posterior measure ${Q^Y}$ through sampling from $\tilde\cP$.
\end{proof}

\subsection{Algorithms for Smoothing and Filtering Problems}

It is worthwhile to point out that the computation cost of  $u_t^{\star}$ and $\pi_t^{\star}$ is demanding. Therefore, we adopt importance sampling with suboptimal solutions. The procedure is essentially the same with the one in~\cite{ZhaTagChe21}, as, in Lemma 1, we already show Gisanonv theorem is applicable in this case. Consider $K$ trajectories $\{g_i^k,\left(\xi_i^k\right)^{\vee}\}$ generated by suboptimal conditions and denote the value of $S_u(0,T)$ over the $k$th trajectory by  $S_u^k(0,T)$. Then the posterior distribution of $(g_t, \xi_t^{\vee})$ is approximated by $\sum_{k=1}^K  w^k \delta_{(g_t^k,\left(\xi_t^k\right)^{\vee})}$, where
\begin{equation}\label{eq:Sw}
	w^k \propto  \frac{d\nu_0}{d\pi_0}(g_0^k,\left(\xi_0^k\right)^{\vee}) \exp[-S_u^k(0,T)],
\end{equation}
and
\begin{equation}\label{eq:Su}
\begin{aligned}
    S_u(t,s) &= \int_{t}^{s} \left[L(\tau,\tilde g_\tau,\tilde\xi_\tau)+\frac{1}{2}\|u_\tau\|^2\right]d\tau \\
    &+ \int_t^s u_{\tau}^{\prime}dW_{\tau}+ \Phi(\tilde g_T,\tilde \xi_T).
\end{aligned}
\end{equation}

In what follows, we divide $[0,T]$ as $0=t_0<t_1<t_2 < \ldots < t_{N}=T $ where for the integer $i$, $t_i-t_{i-1}=\Delta t$, a fixed time horizon. If $\Delta t $ is sufficiently small, then the discretization of \eqref{eq:LieDynamics1} can be written with the exponential mapping. We also apply the one-step Euler-Maruyama approximation to \eqref{eq:LieDynamics2}, which yields the following discrete simulation step,
\begin{subequations}\label{eq:discretetLie}
    \begin{align}
    		g_{i+1} &= g_{i} \exp (\Delta t\xi_i)\\
            \xi_{i+1}^{\vee} &= \xi_i^{\vee} + b(i,g_i,\xi_i)\Delta t  + \sigma_W (u_i\Delta t+\sqrt{\Delta t}\epsilon_i)
    \end{align}
\end{subequations}
Here, $g_0$ and $\xi_0$ are sampled from $\pi_0$, and $\epsilon_i$ follows the standard Gaussian distribution. Note that we replace $t_i$ by $i$ for the ease of presentation. For the observation model, using the same approach, we have
\begin{equation}\label{eq:discretetLie2}
    \begin{aligned}
            Y_{i+1} &= Y_i + h(i,g_i, \xi_i)\Delta t  + \sigma_B \sqrt{\Delta t}\delta_i,
    \end{aligned}
\end{equation}
with $\delta_i \sim N(0,\mathbb{I})$ being a Gaussian random variable.

Next, consider the discrete version of trajectory cost \eqref{eq:Sw}  as follows,
\begin{equation}
    \begin{aligned}
	\hspace{-0.29cm} S_u(0,T) 
      & \approx \sum_{i=0}^{N-1} \left[\frac{1}{2}\|u_i\|^2+\frac{1}{2\sigma_B^2}\|h(i,\tilde g_i,\tilde \xi_i)\|^2\right]\Delta t \\ 
      & - \sum_{i=0}^{N-1}\frac{1}{\sigma_B^2}h(i,\tilde g_i,\tilde \xi_i)^{\prime}(Y_{i+1}-Y_{i}) + u_i^{\prime}\sqrt{\Delta t}\epsilon_i ,
 \end{aligned}
\end{equation}

 where the first equation follows from integration by parts. 

The control-based particle smoothing algorithm  on Lie groups is shown in Algorithm \ref{alg:smoothing}.
\begin{algorithm}
    \caption{Control-based Particle Smoothing on Lie Groups}
    \label{alg:smoothing}
    \begin{algorithmic}
    \STATE \textbf{Input: } 
    $b, \sigma_W, h, \sigma_B$ : Model formulation and parameters
    \STATE $\nu_0$ : Initial distribution
    \STATE $N$: Total number of time points
    \STATE $K$: Total number of particles
    \STATE \textbf{Output:} $\{g^k,\xi^k\}, \{w^k\}$: Smoothing results
    \STATE
    $L,\Phi  \leftarrow$ Formulate the corresponding optimal control problem by \eqref{eq:ginference} and \eqref{eq:costLie}
    \STATE
    $u_t$ and $\pi_0$ $\leftarrow$ Find the suboptimal solutions of \eqref{eq:costLie}
    \FOR{$k \leftarrow 1,\cdots, K$}
            \STATE $\{g_i^k,\xi_i^k\}$ $\leftarrow$ Sample the $k$th trajectory at time point $i$ following the dynamic \eqref{eq:Lie_controlled} with suboptimal conditions  $u_t$ and $\pi_0$
         \STATE $\{w^k\}$ $\leftarrow$ Compute the weights of trajectories by \eqref{eq:Sw}
    \ENDFOR
\end{algorithmic}
\label{alg:A1}
\end{algorithm}
The key difference of smoothing  and filtering is that filtering refers to a length-varying time window while smoothing corresponds to a fixed time window. A naive but computationally-expensive implementation for filtering problem is to recursively run Algorithm \ref{alg:A1} over the increasing time interval. Specifically, for every time point $\tau$, run the smoothing algorithm over $[0,\tau]$ and simulate particles $\{g_{\tau}^k,\xi_{\tau}^k\}_{k=1}^K$ to approximate $P\left((g_{\tau},\xi_{\tau})|Y_t, 0\leq t \leq \tau\right)$.
To reduce the computational demand, we combine the sliding window mechanism proposed in~\cite{ZhaTagChe21} to improve the efficiency of filters. 

The essential component of sliding window algorithm is to involve a sliding window parameter $H$. For the $j$-th time point, if $j\leq H$, then simply run Algorithm 1 over the window $[t_0,t_j]$. Otherwise, solve the smoothing problem over $[t_{j-H},t_j]$ where the initial distribution of $t_{j-H}$, $\nu_{t_{j-H}}$, is recursively updated by the following formulas.
Assume
\begin{equation}\label{eq:priorpre}
		\nu_{t_{j-H-1}} \approx \sum_{k=1}^K w_p^k X_{t_{j-H-1}}^k,
\end{equation}
then 
\begin{equation}\label{eq:prior}
	\nu_{t_{j-H}} \approx \sum_{k=1}^K \hat{w}_p^k X_{t_{j-H}}^k.
\end{equation}
Here $X_{t_{j-H}}^k$ represents the $k$th particle at $t_{j-H}$ sampled by the smoothing algorithm over $[t_{j-H-1},t_{j-1}]$, and $\hat{w}_p^k$ is the normalized weights of $w_p^k \exp[-S_u^k(t_{j-H-1},t_{j-H})]$, where $\exp[-S_u^k(t_1,t_2)]$ stands for the trajectory cost over $[t_1,t_2]$. In the standard particle filtering, a resampling step is performed to reduce weight degeneracy~\cite{DouBri06}. Therefore, we utilize the resampling step as in~\cite{ZhaTagChe21}, which leads to Algorithm \ref{alg_2}, the following control-based filtering algorithm on Lie groups for the case $j>H$. 

\begin{algorithm} 
    \caption{Control-based Particle filtering on Lie Groups}
    \begin{algorithmic}\label{alg_2}
    \STATE \textbf{Input: } 
    $b, \sigma_W, h, \sigma_B,\nu_0, N, K$
    \STATE $j$: Current time point
    \STATE $H$: The Length of the sliding window
    \STATE $\bar{\gamma}$: Threshold of performing resampling
    \STATE $\sum_{k=1}^K w_p^k X_{t_{j-H}}^k$: Prior information at the previous step
    \STATE \textbf{Output:} $\{X^k\}, \{w^k\}$: Filtering results
    $\{w^k_p,X^k_p\}$: Prior information for the next sliding window

    $L,\Phi  \leftarrow$ Formulate the corresponding optimal control problem by \eqref{eq:ginference} and \eqref{eq:costLie} over $[t_{j-H},t_j]$
    \FOR{$k \leftarrow 1,\cdots, K$}
        \STATE $X^k$ $\leftarrow$ Sample the $k$th trajectory  by \eqref{eq:discretetLie} over $[t_{j-H},t_j]$ initialized by $X_{t_{j-H}}^k$
        \STATE $\exp[-S_u^k(t_{j-H},t_j)]$ $\leftarrow$ Compute the weight of $k$th trajectory by \eqref{eq:Sw} over $[t_{j-H},t_j]$.
        \STATE $w^k \propto w^k_p\exp[-S_u^k(t_{j-H},t_j)]$ 
        \STATE $X_p^k = X_{j-H+1}^k$
         \STATE $\{w_p^k\}  \propto \left\{w^k_p\exp[-S_u^k(t_{j-H},t_{j-H+1})]\right\}$
    \ENDFOR 
    \STATE Calculate the effective ratio $\gamma$ of the collection of weights $\{w^k\}_{k=1}^K$ by $\gamma = \frac{1}{K\sum_{i=1}^Kw_i^2}$.
     \IF{$\gamma < \bar{\gamma}$}
                \STATE Sample $\{X_p^k\}$ from \\
                ${\rm multinomial} (\{X_p^k\}, \{w_p^k\}\exp[-S_u^k(t_{j-H+1},t_{j})])$
                \STATE $\{w_p^k\} \propto \exp[S_u^k(t_{j-H+1},t_{j})]$
                \ENDIF
\end{algorithmic}
\end{algorithm}

\section{Experiments}

To demonstrate the efficacy of our algorithm, we consider the following position and velocity estimation problem of a 3D rigid body whose state variable is the tangent bundle of SO(3). To begin with, the system dynamics are given by
\begin{subequations}\label{eq:LieDynamics}
    \begin{align}
        		&dg_t =g_t\xi_tdt \\
                &d\xi_t^{\vee} = \mathbb{M}^{-1}(\mathbb{M}\xi_t^{\vee} \times \xi_t^{\vee} + \mathbb{H}\sigma u_t) dt +  \mathbb{M}^{-1}\mathbb{H}\sigma dW_t\\
                &(g_0,\xi_0) \sim \pi_0
    \end{align}
\end{subequations}

In the above dynamics, the variable $g_t \in $ SO(3) represents the rotation matrix, and the column vector $\xi_t^{\vee} \in \mR^3$ is the body-fixed velocity. We also denote the control signal by $u_t \in \mR^{3}$. Note that for SO(3), $dL_{g_t}(\xi_{t})$ can be simplified to matrix multiplication of $g_t$ and $\xi_t$. The parameter $\mathbb{M} \in \mR^{3 \times 3} $ is the inertia tensor, and $\mathbb{H} \in \mR^{3 \times 3} $ represents the control torques. In addition $\sigma$ is used to modulate the magnitudes of the control signal and noise signal. Here the symbol $\times$ stands for the cross product in the three-dimensional vector space. For $\mathfrak{so}(3)$, we define ${\vee}: \mathfrak{g} \rightarrow \mR^{3}$ by
$$
\begin{bmatrix}
0 & -x_3 & x_2\\
x_3 & 0 & -x_1\\
-x_2 & x_1 & 0
\end{bmatrix}^{\vee}
=
\begin{bmatrix}
x_1 \\
x_2 \\
x_3
\end{bmatrix}.
$$
The deterministic version of the prementioned forward dynamics was initially proposed in~\cite{Cro84} and was later adopted in~\cite{BouThe18a} as an underlying dynamics in optimal control.

For the observation model, we assume the sensors have access to both rotation matrix and the velocity. The rotation matrix $g_t$ is measured by the  accelerometer whose output is  $-g_t^{\prime}r_g$, where $r_g$ is a unit vector aligned with the gravity. Similarly, we also measure $g_t$ with the magnetometer which gives $g_t^{\prime}r_b$ with $r_b$ being a unit vector aligned with the magnetic field (See~\cite{MahHamPfl08} for more discussion of the two sensors). 
In summary, our observation model is  
\begin{equation}
    dY_t = h(t,g_t,\xi_t)dt + \sigma_B dB_t,
\end{equation}
where $h(t,g_t,\xi_t) = (-g_t^{\prime}r_g,g_t^{\prime}r_b,\xi_t^{\vee}) \in \mR^9$.

To define an initial distribution $\pi_0$ that is easy to sample from, we first generate a sample $x \sim N(0,\Sigma_{6 \times 6})$. Then let the first three components $(x_1,x_2,x_3)^{\prime}$ be $\xi_0^{\vee}$ and $g_0$ is given by  $\exp\left([x_4,x_5,x_6]^{\prime}\right)$. This approach is consistent with defining the randomness on SO(3) by latent variables in a vector space and then mapping them to SO(3) via suitable mappings~\cite{FalDeHPim19}.
In this way, the joint distribution of $(g_0,\xi_0^{\vee})$ is well-defined. 

We assume the suboptimal initial distribution $\nu_0$ is always the given one in the prior dynamics. As the control signal $u_t$ is the function of the current state $x_t$, it is sufficient to solve our original optimal control problem \eqref{eq:costLie} by considering the following deterministic and discrete version,
\begin{equation} \label{eq:cost}
    \begin{aligned}
	 \min_{u_t} & \sum_{i=0}^{N-1} \frac{1}{2\sigma_B^2} \Vert h_i \Vert_2^2 + \frac{1}{2} \Vert u_i \Vert_2^2 - \frac{1}{\sigma_B^2} h_i ^{\prime}(Y_{i+1}-Y_i) / \Delta t \\
	& g_{i+1} = g_{i} \exp (\Delta t\xi_i),\\
    & \xi_{i+1}^{\vee}  = \xi_i^{\vee} +\mathbb{M}^{-1}(\mathbb{M}\xi_t^{\vee} \times \xi_t^{\vee} + \mathbb{H}\sigma u_t)\Delta t  \\
    & + \mathbb{M}^{-1}\mathbb{H}\sigma \sqrt{\Delta t}\epsilon_i.
	\end{aligned}
\end{equation}
Here, our ground truth data $Y_i$ is generated by simulation. In the following experiments, we set $\mathbb{M}= {\rm diag}([1,1.11,1.3])$, $\mathbb{H}=\mathbb{I}_3$, $\sigma=1$, $\sigma_B=0.1$,
$\Delta t=0.005$, $N=200$, $K=100$, $\bar{\gamma}=0.1$, $r_g=[0,0,1]$, $r_b=[1/\sqrt{2},0,1/\sqrt{2}]$,
.

The method we adopt to solve this optimal control problem is based on Differential Dynamic Programming over Lie groups (LDDP) proposed in \cite{BouThe18a}. As our framework does not require the exact optimal control signal, we only utilize the linear approximation of the transition function in this paper, which is the iterative linear quadratic regulator (iLQR). We denote this algorithm by PF-iLQR that means the suboptimal solution of \eqref{eq:cost} is given by iLQR. As for the baseline algorithm, the suboptimal control signal is set as zero control, which we denote by PF-ZERO. Note that if $H=1$, then PF-ZERO reduces to the standard Sequential Importance Resampling (SIR) algorithm (See Remark 1 in \cite{ZhaTagChe21}). We measure the quality of the two algorithms with two criteria:  the effective ratio and the error with respect to the simulated ground truth data. 
\begin{figure}
\label{fig:figure1}
\centering
\subfloat[PF-ZERO with varying H ]{\includegraphics[width=0.248\textwidth]{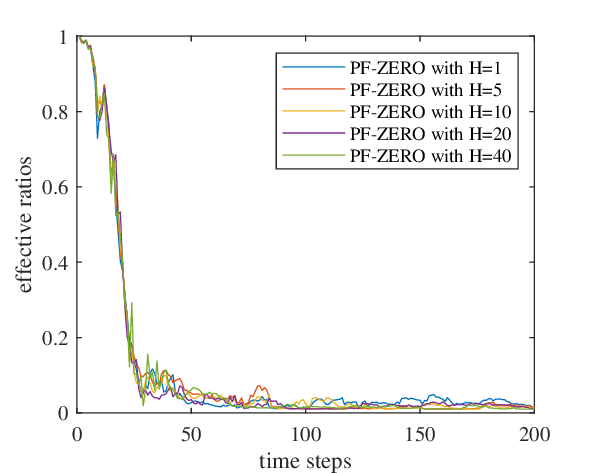}\label{fig:sub1}}
\subfloat[PF-iLQR with varying H]{\includegraphics[width=0.248\textwidth]{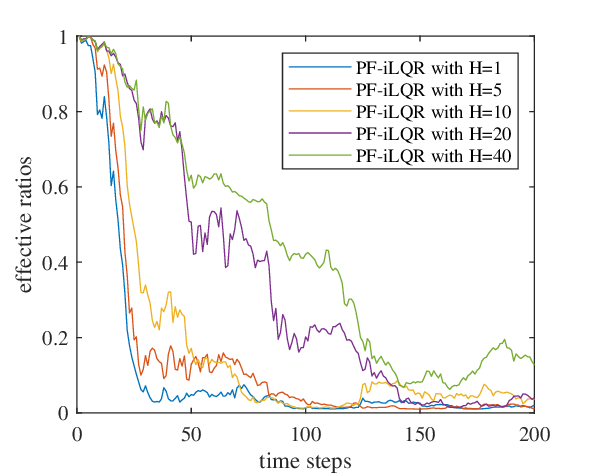}\label{fig:sub2}}
\caption{Comparison of Effective Ratios}
\label{fig:figure1}
\end{figure}

The comparison for the effective ratio is depicted in Figure \ref{fig:figure1}. Clearly, PF-iLQR significantly outperforms PF-ZERO, as the suboptimal signal generated by iLQR would control the evolution of particles to approximate the posterior distribution. We also compare the effects of different $H$ and observe increasing the value of $H$ would in general improve the effective ratios of PF-iLQR. However, very large $H$ may deteriorate the performance. In Table~\ref{table2}, PF-iLQR with $H=200$ that represents using all available data shows slightly worse results than PF-iLQR with $H=40$. Similar result has also been observed in~\cite{ZhaTagChe21}. For PF-ZERO, we observe its performance is nearly irrelevant with H.

In addition to the effective ratios, we also test the errors with respect to the simulated data. For the Lie algebra, $\mathfrak{so}(3)$, we compare the mean squared errors. 
While for SO(3), we follow the procedure used in~\cite{ZhaTagMeh17}, which means computing the rotation angle error  $\delta_{\theta} \in [0,180^{\circ}]$ by the quaternion representation of rotations. More specifically, Let $q^{\star}$ be the unit quaternion of ground-true rotation and $\bar{q}$ be the weighted average of quaternions. Then $\delta_{\theta} = 2\arccos(|\delta q^0|)$ where $\delta q^0$ is the first component of $(\bar{q})^{-1} \times q^{\star}$. In each trial, the error is averaged over time, and we repeat each algorithm for 20 trials when $H$ is larger than 1. The number of repeated trials is 50 for $H=1$ to reduce the impact of randomness.
For completeness, we also compare their performance with resampling steps in Table \ref{table2.5}. Indeed, the improvement with resampling steps is not as significant as the case w/o resampling, as PF-ZERO executes more resampling steps than PF-iLQR. It is worth noting that even in this case, PF-iLQR outperforms PF-ZERO.

\begin{table}[h]
\caption{Comparison of Errors w/o resampling}
\label{table2}
\begin{center}
\begin{tabular}{|c|c|c|}
\hline
 & MSE of $\mathfrak{so}(3)$ & Errors of SO(3) \\
\hline
PF-ZERO with H=1  & 0.5182& 9.0206 \\
\hline
PF-ZERO with H=20 & 0.5230& 8.4491\\
\hline
PF-iLQR with H=1& 0.4710 & 8.8710\\
\hline
PF-iLQR with H=5& 0.3835 & 7.4118\\
\hline
PF-iLQR with H=10& 0.3210 &6.9799\\
\hline
PF-iLQR with H=20& 0.2918 &6.4658\\
\hline
PF-iLQR with H=40& 0.2606 &6.3510\\
\hline
PF-iLQR with H=200& 0.2548 &6.3976\\
\hline
\end{tabular}
\end{center}
\end{table}

\begin{table}[h]
\caption{Comparison of Errors with resampling}
\label{table2.5}
\begin{center}
\begin{tabular}{|c|c|c|}
\hline
 & MSE of $\mathfrak{so}(3)$ & Errors of SO(3)\\
\hline
PF-ZERO with H=1  &0.2934 &7.3829 \\
\hline
PF-ZERO with H=20& 0.2801& 7.0383\\
\hline
PF-iLQR with H=1  &0.2924 &7.6737 \\
\hline
PF-iLQR with H=5& 0.2721 & 6.4915\\
\hline
PF-iLQR with H=10& 0.2799 &6.5156\\
\hline
PF-iLQR with H=20& 0.2648 &6.3164\\
\hline
PF-iLQR with H=40& 0.2544 &6.3311\\
\hline
PF-iLQR with H=200& 0.2533 &6.4026\\
\hline
\end{tabular}
\end{center}
\end{table}

\section{CONCLUSIONS}

In this paper, we generalized the duality between smoothing and optimal control to Lie groups, and building on this duality we developed a particle filtering algorithm for dynamical systems over Lie groups. Our algorithm uses iLQR on Lie groups to approximately solve the dual optimal control problem. The resulting PF-iLQR algorithm significantly improves the effective ratios and thus the overall performance of particle filtering algorithms.

{
\bibliography{./refs}
\bibliographystyle{IEEEtran}
}
\addtolength{\textheight}{-12cm}   

\end{document}